\documentclass[12pt]{amsart}
\usepackage{amsmath}

\def\la{\langle}
\def\ra{\rangle}
\def\bee{\begin{equation*}}
\def\eee{\end{equation*}}

\def\e{\epsilon}
\def\lf{\left}
\def\heat{\lf(\frac{\p}{\p t}-\Delta\ri)}
\def\ijb{i\bar{j}}
\def\ri{\right}
\def\hg{\hat{g}}
\def\a{{\alpha}}

\def\tphi{{\tilde{\phi}}}

\def\p{\partial}
\newcommand\R{{\mathbb R}}

\newcommand\C{{\mathbb C}}

\def\ii{\sqrt{-1}}
\def\jbar{{\bar\jmath}}

\def\tr{\tilde{R}_{ij}}
\def\K{K\"ahler }
\def\KR{K\"ahler-Ricci }

\def\A{Amp\`{e}re }

\def\ddbar{\partial\bar\partial}
\def\be{\begin{equation}}
\def\ee{\end{equation}}

\def\lf{\left}
\def\ri{\right}
\def\a{{\alpha}}

\def\txi{\widetilde \xi}

\def\e{\epsilon}

\def\ijb{{i\jbar}}

\def\Rm{\text{\rm Rm}}

\def\p{\partial}

\def\tr{\text{\rm tr}}

\def\C{\Bbb C}

\def\p{\partial}

\def\p{\partial}

\def\C{\Bbb C}
\def\ii{\sqrt{-1}}

\newtheorem{thm}{Theorem}[section]

\newtheorem{lem}{Lemma}[section]
\newtheorem{prop}{Proposition}[section]
\newtheorem{cor}{Corollary}[section]
\theoremstyle{definition}

\theoremstyle{remark}
\newtheorem{rem}{Remark}
\numberwithin{equation}{section}
 
\begin{document}
\title{Deforming complete Hermitian metrics with unbounded curvature}
\author{Albert Chau$^1$}
\address{Department of Mathematics,
The University of British Columbia, Room 121, 1984 Mathematics
Road, Vancouver, B.C., Canada V6T 1Z2} \email{chau@math.ubc.ca}

\author{Ka-Fai Li}

\address{Department of Mathematics,
The University of British Columbia, Room 121, 1984 Mathematics
Road, Vancouver, B.C., Canada V6T 1Z2} \email{kfli@math.ubc.ca}

\author{Luen-Fai Tam$^2$}

\thanks{$^1$Research
partially supported by NSERC grant no. \#327637-06}
\thanks{$^2$Research partially supported by Hong Kong RGC General Research Fund  \#CUHK 403108}

\address{The Institute of Mathematical Sciences and Department of
 Mathematics, The Chinese University of Hong Kong,
Shatin, Hong Kong, China.} \email{lftam@math.cuhk.edu.hk}
\thanks{\begin{it}2000 Mathematics Subject Classification\end{it}.  Primary 53C55, 58J35.}

\begin{abstract} We produce solutions to
the K\"ahler-Ricci flow emerging from complete initial metrics $g_0$ which are $C^0$ Hermitian limits of \K metrics.  Of particular interest is when $g_0$ is \K  with unbounded curvature. We provide such solutions for a wide class of $U(n)$-invariant \K metrics $g_0$ on $\C^n$, many of which having unbounded curvature.  As a special case we have the following Corollary: The \KR flow \eqref{krf} has a smooth short time solution starting from any smooth complete $U(n)$-invariant \K metric on $\C^n$ with either non-negative or non-positive holomorphic bisectional curvature, and the solution exists for all time in the case of non-positive curvature.

\end{abstract}

\maketitle\markboth{Albert Chau, Ka-Fai Li and Luen-Fai Tam} {On the existence of K\"ahler-Ricci flow of metrics with unbounded curvature}
\section{Introduction}

Let $(M^n,g_0)$ be a complete noncompact Riemannian manifold. The Ricci flow is the following evolution equation
\be\label{rf}
 \left\{
   \begin{array}{ll}
     \displaystyle\frac{\partial}{\partial t} g_{ij} =-2R_{ij} \\
     g(0)  = g_0.
   \end{array}
 \right.
 \ee
 In \cite{Shi1}  Shi proved  that if the curvature of $g_0$ is bounded then \eqref{rf} has a solution $g(t)$ up to some time $T>0$ depending only on the curvature bound for $g_0$ and the   dimension $n$ of $M$ such that the curvature is bounded in space-time. If in addition that $(M^n,g_0)$ is a \K manifold with complex dimension $n$, then Shi \cite{Shi2} proved that the solution $g(t)$ is also K\"ahler. Hence $g(t)$ satisfies \KR flow   equation:

 \be\label{krf}
 \left\{
   \begin{array}{ll}
     \displaystyle\frac{\partial}{\partial t} g_{\ijb} =-R_{\ijb} \\
     g(0)  = g_0.
   \end{array}
 \right.
 \ee
See   Theorem \ref{shishortime} for more details.

There are many results of existence without assuming  that the initial condition $g_0$ has bounded curvature. In \cite{Si}, Simon proved that   starting from any sufficiently small $C^0$ perturbation $g_0$ of a complete Riemannian metric with bounded curvature, there is a short time solution of the Ricci harmonic heat flow.  We also refer to the works \cite{KL, SS1} where the Ricci harmonic heat  flow is solved starting with rough initial data obtained from a sufficiently small perturbation of the Euclidean metric on $\R^n$, and \cite{SS2} for a similar result for the hyperbolic metrics. In \cite{CW}, Cabezas-Rivas and Wilking obtained a short time existence result of the Ricci flow  starting from any complete Riemannian metric with nonnegative complex sectional curvature. They do not assume the curvature is bounded and do not assume the initial metric is a small perturbation   of a complete metric with bounded curvature. The solutions from \cite{Si}, \cite{KL},\cite{SS1} and \cite{SS2}
are complete and have bounded curvature when $t>0$.  In \cite{CW}, complete solutions are constructed where the curvature is bounded whenever $t>0$ and examples are also given of complete solutions where the curvature is unbounded when $t>0$.

For \KR flow, when $n=1$, Geisen-Topping \cite{GT} proved that \eqref{krf} always has a solution starting from any smooth \K metric $g_0$ which may have unbounded curvature, and may even be incomplete.  In fact, they also constructed solutions where $g(t)$ is complete with unbounded curvature for all $t\in [0, T)$.
Using the construction of \cite{CW}, Yang-Zheng proved that if $g_0$ is a $U(n)$ invariant complete \K metric with nonnegative sectional curvature, and with some technical assumptions on the solution $g(t)$ of \eqref{rf}, then $g(t)$ is \K for $t>0$. Hence in this case \KR flow \eqref{krf} has short time solution.

In this work we want to discuss the short time existence and long time existence of the \KR flow \eqref{krf} in higher dimensions without the assumption that $g_0$ has bounded curvature. We obtain the following:

\begin{thm}\label{t-1-intro} Let $g_0$ be a complete continuous Hermitian metric on a noncompact complex manifold $M^n$. Suppose there exists a sequence  $\{h_{k,0}\}$   of smooth complete \K metrics with bounded curvature on $M$
converging
uniformly on compact subsets to $g_0$ and another complete \K metric $\hg$ with bounded curvature on $M$ such that  $C^{-1} \hat{g} \leq h_{k,0} \leq C \hat{g}$ for some $C$ independent
of $k$.  Then for some $T>0$, the \KR flow \eqref{krf} has a complete smooth solution $g(t)$ on $M\times (0, T)$ which has bounded curvature for all $t>0$, and extends continuously to $M\times[0,T)$ with $g(0)=g_0$.
Moreover, if $g_0$ is smooth and $\{h_{k,0}\}$ converges smoothly and uniformly on
compact subsets of $M$, then $g(t)$ extends to a smooth solution to \eqref{krf} on $M\times
[0, T)$ with $g(0)=g_0$.
\end{thm}
One can also estimate the existence time $T$ and bounds of the norms of the curvature tensor with its covariant derivatives of $g(t)$, see Theorem \ref{thmgequivalenttoboundedcurvature} for more details.

As a corollary, we obtain an estimate of $T$ in Theorem \ref{shishortime} in terms of the upper bound of the holomorphic bisectional curvature. In fact, we can prove a more general result (see Corollary \ref{existencetimeestimate-2}):

\vspace{10pt}
\begin{it} Let $(M^n,g_0)$ be a complete noncompact \K manifold with bounded curvature.  Suppose that $\hg\leq g_0\le C\hg$ for some complete \K metric  $\hg$ with bounded curvature  and holomorphic bisectional curvatures bounded above by $K$.   Let $T=1/(2nK)$ if $K>0$, otherwise let $T=\infty$.  Then the \KR flow \eqref{krf} has a complete smooth solution $g(t)$ on $M\times[0,T)$ with $g(0)=g_0$.  \end{it}\vspace{10pt}

Another corollary is that one can prove that the \KR flow \eqref{krf} has short time solution if  $g_0$ is perturbation of a complete \K metric $\hg$ with bounded curvature by a potential satisfying certain growth conditions.  More precisely (see Corollary \ref{perturbpotential}),

\vspace{10pt}
\begin{it} Let $(M^n,g_0)$ be a complete noncompact \K manifold with bounded curvature.  Suppose $u$ is a $C^2$ function such that $|\nabla u|_{g_0}$ and $|u|$ are of sublinear growth and such that $g_0+\ii\p\bar\p u$ is uniformly equivalent to $g_0$.
Then for some $T>0$, the \KR flow \eqref{krf} has a complete smooth solution $g(t)$ on $M\times (0, T)$ which has bounded curvature for all $t>0$ and extends continuously to $M\times[0,T)$ with $g(0)=g_0+\sqrt{-1}\partial\bar{\p} u$.  \end{it}
\vspace{10pt}

 If the condition in Theorem \ref{t-1-intro} that $ C^{-1}\hg\le h_{k,0}\le  C\hg$ is relaxed to only assuming   $h_{k,0}\ge C^{-1}\hg$, we still can obtain a short time solution  under some  additional assumptions on $h_{k,0}$. See Theorem  \ref{thmgboundedbelowbyboundedcurvature}, for more details. However, in this case, we do not know if the curvature of the solution  is bounded for $t>0$.

  Applying the general existence theorems to $U(n)$ invariant \K metrics on $\C^n$, we obtain:
\begin{thm}\label{t-2-intro}
Let $g_0$ be a complete smooth $U(n)$-invariant \K metric on $\C^n$ with either non-negative or non-positive holomorphic bisectional curvature.  Then for some $T>0$, the \KR flow \eqref{krf} has a complete smooth  $U(n)$-invariant solution $g(t)$ on $\C^n\times [0, T)$ with $g(0)=g_0$.  Moreover,  the solution  exists for all time in the case of non-positive holomorphic bisectional curvature.
\end{thm}
This gives an affirmative answer to a question posed by Yang-Zheng \cite{YZ}. In fact, one can prove results more general than the Theorem above. See Theorems \ref{mainthm-a}, \ref{mainthm} and their corollaries for more details. We also obtain some long time existence results for $g_0$ with nonnegative holomorphic bisectional curvature, see Theorem \ref{longtimenonnegativecurvature}.

The organization of the paper is as follows. In section \S2 we review some basic theory and estimates for \eqref{krf}, and in \S3 we prove some further a  priori estimates which we will need later.  \S4 contains our main existence theorems Theorems \ref{thmgequivalenttoboundedcurvature} and \ref{thmgboundedbelowbyboundedcurvature} and accompanying corollaries.  In \S5 we review Wu-Zheng's description in \cite{WZ} of $U(n)$ invariant \K metrics on $\C^n$ and use this to apply our previous results to prove Theorems \ref{mainthm-a} and \ref{mainthm}.

The first author would like to thank Bo Yang, and the third author would like to thank Fangyang Zheng for helpful discussions and interest in this work.

\section{Preliminaries}
 In this section, we review some well known results for the \K Ricci flow which we will use in the paper.  We first recall Shi's short time existence theorem for \eqref{krf} in \cite{Shi2}.

 \begin{thm}\label{shishortime} Let $(M^n,g_0)$ be a complete noncompact \K manifold with curvature bounded by a constant $K$.  Then
for some $0<T\leq \infty$ depending only on $K$ and the dimension $n$, there exists 
a smooth solution $g(t)$ to \eqref{krf} on $M \times [0, T)$ with $g(0)=g_0$ such that
 \begin{enumerate}
 \item[(i)] $g(t)$ is \K and equivalent to $g_0$ for all $t\in [0, T)$;
 \item[(ii)] $g(t)$ has uniformly bounded curvature on $M\times[0,T']$ for all
$0<T'<T$.  In particular, for any $l\geq 0$ there exists a constant $C_l$
depending only on $l$, $g_0$ and the dimension $n$ such that
     $$
 sup_M |\nabla^l \Rm(h(t))|^2_{h(t)}\le \frac {C_l}{t^{l}},
 $$
 on $M\times[0,T']$.
     \item[(iv)] If $T<\infty$ and $\displaystyle \lim_{t\to T} \sup_{M} |Rm(x,
t)|<\infty$, then $g(t)$ extends to a smooth solution to \eqref{krf} on $M \times [0, T_1)$ for some $T_1>T$ so that
(ii) is still true with $T$ replaced by $T_1$.
         \end{enumerate}
\end{thm}

 The solution $g(t)$ in Theorem \ref{shishortime} has uniformly  bounded curvature on $M\times[0,T')$ for any $0<T'<T$. From this it is easy to see that $g(0)$ and $g(t)$ are uniformly equivalent on $M\times[0,T')$. On the other hand, it is  a well known fact one can study the \KR flow   \eqref{krf} through the parabolic Monge-\A equation. This was originated in \cite{C}, and we refer to \cite{CT} and references therein for further details on this fact.  By  the Evans-Krylov theory \cite{Kr,Ev} for fully non-linear equations, which in the case of \eqref{krf}  takes the form in Theorem \ref{t-EvanKrylov} below, one may conclude that if $g(0)$ and $g(t)$ are uniformly equivalent, then subsequent curvature bounds follow (see \cite{SW}, and also \cite{Yu} in the complete case, for proofs using only the maximum principle). Actually, we need a more general version in the sense that we need local estimates and $g(t)$ is assumed to be uniformly equivalent to a fixed background metric $\hat g$.

 Let us first fix some notations and terminology. $(M^n,\hat g)$ is said to have   bounded geometry of infinite order if the curvature tensor and all its covariant derivatives are uniformly bounded. In particular, the solution $g(t)$ in Theorem \ref{shishortime} has bounded geometry of infinite order for $t>0$.

Also, we will denote the geodesic ball with respect to the metric $g$ with center at $p$ and radius $r$ by $B_g(p,r)$. The following theorem can be found in \cite{SW}.

 \begin{thm}\label{t-EvanKrylov}  Let $(M^n,\hat g)$ be a complete noncompact
\K manifold with bounded geometry of infinite order. Let $h(t)$ be a solution
of \KR \eqref{krf} on $M\times[0,T)$ with initial condition $h_0$ which is a
complete \K metric.  For any $x\in M$, suppose there is a constant $N>0$, such
that
 $$
 N^{-1}\hat g\le h(t)\le N\hat g
 $$
 on $ B_{\hat g}(x,1)\times [0,T)$.  Then

  \begin{itemize}
    \item [(i)]
 $$
 | \hat{\nabla}^k h|^2_{\hat g}\leq \frac{C_k}{t^{k}}
 $$
 on $  B_{\hat g}(x,1/2)\times(0, T)$, for some constant $C_k$ depending only on $k$,
$\hat g$, $n$, $T$ and $N$.

 \item[(ii)] If we assume $|\hat{\nabla}^k h_0|_{\hat g}^2$ is bounded in $
B_{\hat g}(x,1)$ by $c_k$, for $k\ge 1$,  then
  $$
 |\hat\nabla ^k h|^2_{\hat g}\le {C_k} ,
 $$
  on $  B_{\hat g}(x,1/2)\times[0, T)$ for some constant $C_k$ depending only on $k$,
$c_k$, $n$, $T$ and $N$.
  \end{itemize}
\end{thm}

 \begin{proof} Since $\hat g$ has bounded geometry of infinite order, by \cite{TY}, for any
$x\in M$ there exists a local biholomorphism $\phi_x: D  \to M$, where $D=D(1)$
is the open unit ball in $\C^n$, satisfying the following in $D$
 \begin{enumerate}
 \item [(a)] $\phi_x(0)=x$, $\phi_x(D)\subset \hat B(x,1)$, $\phi_x(D)\supset
\hat B(x,2\delta)$ for some $\delta>0$ which is independent of $x$.
\item [(b)] $C^{-1}\delta_{i\jbar} \leq (\phi_x^*(\hat{g}))_{i\jbar} \leq C
\delta_{i\jbar}$ for some $C$ independent of $x$.
\item [(c)] $\lf|\dfrac{\partial^l (\phi_x^*(\hat{g}))_{i\jbar}}{\partial
z^L}\ri| \leq C_{l}$  for any $l, i, j$ and multi index $L$ of length $l$ for
some constant $C_l$ which is independent of $x$. \end{enumerate}

 Consider $\phi_x^*(h(t))$, which clearly will solve \ref{krf} on $D(1)
\times[0, T)$. By the Evans-Krylov theory \cite{Ev}, \cite{Kr} for fully
non-linear elliptic and parabolic equations (see also \cite{SW} for a maximum
principle proof in the case of \K Ricci flow), the result follows.
 \end{proof}

 We end this section with the following longtime existence Theorem from \cite{CT} when we look at certain special solutions to \eqref{krf} on $\C^n$ in Theorem \ref{longtimenonnegativecurvature}.

 \begin{thm}\label{longtimeexistence}
Let $(M, g_0)$ be a complete non-compact \K manifold such that
\begin{enumerate}
\item [i)] $|\Rm(x)| \to 0$ as $d(x)\to \infty$ where $d(x)$ is the distance function on $M$ from some $p\in M$.
\item [ii)] The injectivity radius of $(M, g_0)$ is uniformly bounded below by some constant $c>0$.
\item[iii)] There exists a strictly pluri-subharmonic function $F$ on $M$.
 \end{enumerate}
 Then the \KR flow \eqref{krf} has a complete smooth solution $g(t)$ on $M\times[0, \infty)$ with $g(0)=g_0$.  Moreover, the curvature of $g(t)$ is bounded uniformly on $M\times[0,T]$ for all $T<\infty$.
 \end{thm}

\section{Further estimates}

In this section we prove some further estimates which we will require to prove our main theorems. One main tool is  Theorem \ref{t-EvanKrylov}. Hence we want to obtain $C^0$ estimate of solution of \KR flow in terms of a background metric.

Recall that the holomorphic bisectional curvature of a \K manifold is said to be bounded above by $K$ if
\begin{equation}\label{holobi-def}
 \frac{R(X,\bar X,Y,\bar Y)}{\|X\|^2\|Y\|^2+|\la X,\bar Y\ra|^2}\le K
 \end{equation}
 for any two nonzero (1,0)-vectors $X,Y$. The holomorphic bisectional curvature of a \K manifold being bounded below by $K$ is defined similarly.

 \begin{lem}\label{lemgequivalenttoboundedcurvature}
Let $h(t)$ be a solution to \eqref{krf} on $M^n\times [0,
T_0)$ with $h(0)=h_0$ such that $h(t)$ has uniformly bounded curvature for on $M\times [0, T'] $ for all $0<T'<T_0$.  Let $\hat g$ be another  complete  \K metric on $M$ with bounded curvature such that  the  holomorphic bisectional curvature bounded above by $K$. Let $T=\frac1{2nK}$ if $K>0$, otherwise let $T=\infty$.
\begin{itemize}
  \item [(i)] Suppose $h_0\ge \hat g$. Then
  $
  h(t)\ge \lf(\frac1n-2Kt\ri)\hat g
  $
  on $M\times[0,\min\{T_0,T\})$.
  \item [(ii)] Suppose in addition to (i) we have $h_0\le C\hat g$, that is, suppose $\hat g\le h_0\le C\hat g$, then
  $$
  (1-w(t))\hat{g}\leq h(t)\le (1+w(t))\hat g$$
 on $M\times[0,\min\{T_0,T\})$,
\end{itemize}
where $w(t)=\sqrt{v_2(t)(v_1(t)+v_2(t)-2n)},$
 $$v_1(t)=\frac{1}{\frac1n-2Kt} \hspace{5pt}, v_2(t)=nCe^{-2\kappa v_1(t)t}$$
and $\kappa$ is a lower bound on the bisectional curvature of $\hat{g}$.  In particular, we have $w(0)=n\sqrt{C(C-1)}.$

\end{lem}
\begin{proof}
(i)
Let  $\phi(t):= \tr_{h(t)} \hat{g}$.
Let $\Box=\frac{\p}{\p t}-\Delta$, where $\Delta$ is the Laplacian with respect
to $h(t)$.  Then as in \cite{ST}, we can calculate in a normal coordinate relative to $h(t)$ and use \eqref{krf} to get
\begin{equation}
\begin{split}
\Box \phi =&  ((h_t)^{i\jbar}\hat{g}_{i\jbar})-h^{k\bar{l}}(h^{i\jbar}\hat{g}_{i\jbar})_{k\bar{l}}\\
=&  (R^{i\jbar}\hat{g}_{i\jbar})- (R^{i\jbar}\hat{g}_{i\jbar})+h^{k\bar{l}}h^{i\bar{j}}\widehat{R}_{i\bar{j}k\bar{l}}-\widehat{g}^{p\bar{q}}h^{k\bar{l}}h^{i\bar{j}}\partial_{k}\widehat{g}_{i\bar{q}}\partial_{\bar{l}}\widehat{g}_{p\bar{j}}\\
& \leq 2K\phi^{2}.
\end{split}
\end{equation}
Now $v_1(t)$ is the positive solution to the
ODE $$\frac{dv_1(t)}{dt}=2Kv_1^{2}(t); \hspace{5pt} v_1(0)=n$$
for $t\in[0,T)$. Let $S\in (0, \min\{T_0,T\})$ be fixed. Since $h(t)$ has uniformly bounded curvature on $M\times [0, S]$ we have $h(t)\ge C_1h_0\ge C_1\hg$ for some $C_1>0$ and hence
$\phi$ is a bounded function on $M\times [0, S]$. Moreover, $v_1$ is also a bounded function on $M\times[0,S]$. Let $A=\sup_{M\times[0,S]}(\phi+v_1)$. Then on $M\times[0,S]$
\begin{eqnarray*}
\Box \lf(e^{-(2AK+1)t}(\phi-v_1)\ri) &\\
\leq  e^{-(2AK+1)t}&\lf[2K(\phi^{2}-v_1^{2})-(2AK+1)(\phi-v_1)\ri] \\
 =  e^{-(2AK+1)t}&\lf[2K (\phi+v_1)-(2AK+1)\ri](\phi-v_1)
 \end{eqnarray*}
which is nonpositive at the points where $\phi-v_1\ge0$. Using the fact that $h(t)$ has uniformly bounded curvature on $M\times[0,S]$ and the fact that $e^{-(2AK+1)t}(\phi-v_1)\le0$ at $t=0$, which is uniformly bounded on $M\times[0,S]$,  we conclude that $e^{-(2AK+1)t}(\phi-v_1)\le0$ and thus $(\phi-v_1)\le0$ on $M\times[0,S]$ by the maximum principle, see \cite[Theorem 1.2]{NT} for example.
This proves (i).

(ii) Let $\psi(t):= \tr_{\hat{g}} h(t).$ For any fixed $S\in [0, min\{T_h, T\})$, as in \cite{C} we calculate in a normal coordinate relative to $\hat{g}$ and use \eqref{krf} to get  that on $M\times[0, S)$:

\begin{equation}
\begin{split}
\Box \psi= & (\hat{g}^{i\jbar}(h_t)_{i\jbar})-h^{k\bar{l}}(\hat{g}^{i\jbar}h_{i\jbar})_{k\bar{l}}\\
= &-(\hat{g}^{i\jbar}R_{i\jbar})-h^{k\bar{l}}(\hat{R}^{i\jbar}_{k\bar{l}}h_{i\jbar})+(\hat{g}^{i\jbar}R_{i\jbar})-\widehat{g}^{i\bar{j}}h^{p\bar{q}}h^{k\bar{l}}\partial_{i}h_{p\bar{l}}\partial_{\bar{j}}h_{k\bar{q}}\\
=&-h^{k\bar{l}}h_{i\bar{j}}\widehat{R}_{k\bar{l}}^{\,\,\bar{j}i}-\widehat{g}^{i\bar{j}}h^{p\bar{q}}h^{k\bar{l}}\partial_{i}h_{p\bar{l}}\partial_{\bar{j}}h_{k\bar{q}}\\
 & \leq  -2\kappa v_1(t)\psi\\
 &\leq -2\kappa v_1(S)\psi\\
\end{split}
\end{equation}
by (i).
 Let $w_S(t)=nCe^{-2cv_1(S)t}$ be the solution to the ODE $$\frac{dw_S(t)}{dt}=-2cv_1(S)w_S(t); \hspace{5pt} w_S(0)=nC.$$
  Then arguing as above, we have $\psi\leq w_S$ on $M^{n}\times[0,S]$.  In
particular, we get $\psi(S)\leq w_S(S)$ for every $S\in [0, min\{T_0, T\})$.

 So far, we have $\phi(t)\leq v_1(t)$, and $\psi(t)\leq v_2(t)$ on $M\times [0,
min\{T_0, T\})$ where $v_1, v_2$ are as in the statement of the Lemma.  Now we follow an idea from \cite{Shi1}.  At any point in
$(p, t)\in M\times[0, min\{T_0, T\})$, let $\lambda_i's$ be the eigenvalues
of $h$ with respect to $\hat{g}$, and calculate at $(p, t)$

\begin{equation}\label{equivancecloseness}
\begin{split}
\displaystyle \sum_{i=1}^n \frac{1}{\lambda_i} (1-\lambda_i)^2=&\sum_{i=1}^n
\frac{1}{\lambda_i} +\lambda_i -2\\
\le&\phi+\psi-2n\\
\le&v_1(t)+v_2(t)-2n\\
\end{split}
\ee
and thus for any fixed $i$ we have
\be
-w(t) \leq \lambda_i-1\leq w(t)
\ee
where $w(t)=\sqrt{v_2(t)(v_1(t)+v_2(t)-2n)}$.  The conclusion in (ii) then
follows.

\end{proof}

 The following lemma basically says that if a local solution $h(t)$ to
\eqref{krf}  is a priori uniformly equivalent to a fixed metric $\hat{g}$ in
space time, and close to $\hat{g}$ at time $t=0$, then it remains close to
$\hat{g}$
in a uniform space time region.  Note that in contrast to Lemma
\ref{lemgequivalenttoboundedcurvature}, the a priori assumption here is on
$h(t)$ for all $t$.

\begin{lem}\label{lemgequivalenttoboundedcurvaturelocal}
Let $h(t)$ be a smooth solution to \eqref{krf} on $B(1)\times [0,
T)$ with $h(0)=h_0$ where $B(1)$ is the unit Euclidean ball in $\C^n$. Let
$\hat g$ be a smooth \K metric on $B(1)$.   Suppose
\be\label{lemgequivalenttoboundedcurvaturelocale1}
  N^{-1}\hat{g}\leq h(t)\le N\hat g
\ee
on $B(1)\times [0,T)$ for some $N>0$, and that
\be\label{lemgequivalenttoboundedcurvaturelocale1}
 \hat{g}\leq h_0\le C\hat g
\ee
on $B(1)$.  Then there exists a positive continuous function $a(t):[0, T)\to \R$
depending only on $\hat{g}, N, C$ and $n$ such that
\be\label{lemgequivalenttoboundedcurvaturelocale1}
 \frac{(1-a(t))}{C}h_0\leq h\le (1+a(t))h_0
\ee
on $B(1/2) \times[0,T)$, where $a(0)=n\sqrt{C(C-1)}$.
\end{lem}
\begin{proof}
As in the previous Lemma, let $\phi=\tr_h \hat{g}$, $\psi=\tr_{\hat{g}} h$ on
$B(1)\times [0,T_0)$.  Choose some smooth non-negative cutoff function on
$\eta:B(1)\to \R$ satisfying $\eta|_{B(1/2)}=1$, $\eta|_{(B(3/4))^c}=0$,
$|\hat{\nabla} \eta|^2 \leq C_1 \eta$, $|\p\bar\p\eta|_{\hat g}\leq C_2 $ on
$B(1)$ for some constants $C_1, C_2$ depending only on $\hat g$.  Using the fact that $h(t)\ge N^{-1}\hat g$, we have
$$|\nabla \eta|^2=h^{\ijb}\eta_i\eta_{\bar j}\le N|\hat\nabla\eta|^2\le NC_1,$$ and $$|\Delta\eta|=\lf|h^{\ijb}\eta_{\ijb}\ri|\le N|\p\bar \p\eta|_{\hat g}\le NC_2.$$

Now we consider the function $\eta\phi$
on $B(1)\times [0,T)$.  Then in
$B(1)\times [0,T)$ at the point where  $\eta>0$, as in the proof    of Lemma \ref{lemgequivalenttoboundedcurvature} (i) we obtain
\be
\begin{split}
(\partial_t-\Delta)(\eta\phi)&=\eta(\partial_t-\Delta )\phi-2<\nabla\eta,\nabla
\phi> -\phi \Delta \eta\\
&=\eta(\partial_t-\Delta)\phi-2\frac{<\nabla\eta,\nabla
(\eta\phi)>}{\eta}+\frac{2|\nabla \eta|^2}{\eta}\phi -\phi \Delta \eta\\
&\leq  \eta C_3\phi^2 -2\frac{<\nabla\eta,\nabla
(\eta\phi)>}{\eta}+2NC_1\phi+NC_2 \phi\\
&\leq C_4-2\frac{<\nabla\eta,\nabla (\eta\phi)>}{\eta}\\
\end{split}
\ee
where the constants $C_3, C_4$ depend only on $\hat{g}, N, C$ and $n$, where we have used the assumption \eqref{lemgequivalenttoboundedcurvaturelocale1}. Since $\eta\phi$ is zero outside $B(3/4)$, applying    the maximum principle to $\eta\phi-C_4t$ one can conclude that
$$\eta\phi\leq n+  C_4t=:\tilde{v_1} (t)$$
on $B(1)\times [0,T)$.

Now consider the function $\eta\psi$ on $B(1)\times [0,T)$.  Using the proof of Lemma
\ref{lemgequivalenttoboundedcurvature} (ii) and estimating as above we obtain
\begin{equation}
\eta\psi\leq nC +  C_5t=:\tilde{v_2} (t)
\ee
on $B(1)\times [0,T)$ for som constants $C_5$ depending only on $\hat{g}, N, C$
and $n$.

Now at any point in $(p, t)\in B(1/2)\times [0,T)$, let $\lambda_i's$ be the
eigenvalues of $h$ with respect to $\hat{g}$.  Then as in the proof of Lemma
\ref{lemgequivalenttoboundedcurvature} (ii) we get that at $(p, t)$

\be
-\tilde{w}(t) \leq \lambda_i-1\leq \tilde{w}(t)
\ee
where $\tilde{w}(t)=\sqrt{\tilde{v_2}(t)(\tilde{v_1}(t)+\tilde{v_2}(t)-2n)}$.
Since $\tilde{v_1} (0)=n$ and $\tilde{v_2} (0)=nC$, the lemma  follows easily from this.
\end{proof}

 In contrast to the previous lemma, in the following lemmas we only assume a lower bound on a solution $h(x,t)$ to \eqref{krf}.

\begin{lem}\label{lemgboundedbelowboundedcurvature} Let $h(x,t)$ be a smooth solution to \eqref{krf} on $M\times[0,T)$ with $h(0)=h_0$. Let $p\in M$. Suppose
there is a positive continuous function $\a(t):[0, T)\to \R$ such that
 $$
 h(t)\ge \a(t)\hat g.
 $$
where $\hat{g}$ is a complete \K metric with bounded curvature.  Then,
there exists a positive continuous function $\beta(r,t):[1, \infty)\times [0,T)\to \R$
$\beta(r, t)$ depending
only on $\hat g$ the upper bound of $\tr_{\hat g}h_0$ in $B_{\hat{g}}(p,2 r)$, the lower
bound of scalar curvature $R(0)$ of $h(0)$ in $B_{\hat{g}}(p, 2r)$, $\a(t)$
and the dimension $n$ such that for $r\ge 1$
$$
h(t)\le \beta(r, t) \hat g.
$$
in $B_{\hat{g}}(p, r)\times [0,T)$.

\end{lem}
\begin{proof}   
 Let $d(x)$ be the distance with respect to $\hat g$ from $x$ to a fixed point $p\in M$.  Since $\hat g$ has bounded curvature, by \cite{Shi2} (see also \cite{T}) there exists a smooth positive function $\rho(x)$ satisfying $d(x)+1\leq \rho(x)\leq d(x)+C$ on $M$ for some $C>0$, with $|\hat\nabla \rho|$, $|\hat\nabla^2\rho|$ are bounded on $M$.  Hence without loss of generality, we may assume for simplicity that $d(x)$ is in fact smooth with $|\hat\nabla d|, |\hat \nabla^2d|$ bounded on $M$. 

Let $\phi(s)$ be smooth function on $\R$ such that $\phi=1$
for $s\le 1$ and is zero for $s\ge2$. Moreover, we assume $\phi'\le0$,
$(\phi')^2/\phi\le C_1$, $|\phi''|\le C_2$. Let $R$ be the scalar curvature of
$h(t)$. Then
\be
\heat R\ge \frac1n R^2.
\ee
on $M\times[0 ,T)$.  Let $\varphi(x)=\phi(d(x)/r)$. Then $\varphi(x)=0$ if $d(x)\ge 2r$. Fix some $T'<T$.  Then as in the proof of the previous lemma,  we compute
\bee
\begin{split}
| \nabla \varphi|^2=&\frac1{r^2}(\phi')^2|\nabla d|^2\\
=&\frac1{r^2}(\phi')^2h^{\ijb}d_i d_{\bar j}\\
\le&\frac1{ r^2\a(t)}(\phi')^2 \hat g^{\ijb}d_i d_{\bar j}\\
\le&\frac{C_3}{r^2}(\phi')^2
\end{split}
\eee
on $B(2r)\times[0, T']$ for some constant $C_3$ depending only on $T', \a(t)$ and $\hat{g}$.
Similarly,
\bee
\begin{split}
| \Delta \varphi|=&|\frac1r\phi'\Delta d+\frac1{r^2}\phi''|\nabla d|^2|\\
\le & C_4(\frac1r+\frac1{r^2})
\end{split}
\eee
 on $B(2r)\times[0, T']$ where $C_4$ depends on $C_1, C_2, T', \a(t)$ and $\hat{g}$.

 Now
 \be
 \begin{split}
 \heat (\varphi R)=&\varphi \heat R-R\Delta \varphi-2\la\nabla R,\nabla
\varphi\ra\\
 \ge&\frac1n\varphi R^2-C_5|R|- 2\la\nabla R,\nabla \varphi\ra
 \end{split}
 \ee
on $B(2r)\times[0, T']$ where $C_5$ depends only on $C_4$ and $r$. Suppose the infimum of $\varphi R$ on $B(2r)\times[0, T']$ is
attained at $t=0$, then $R\ge \min\{0,\inf_{B_{\hat{g}}(p, 2r)}R(h_0)\}$ on $B_{\hat{g}}(r)$.
Suppose instead that $\varphi R$ attains a negative minimum at some $(x, t) \in B(2r)\times[0, T']$ where $t>0$. Then at
$(x,t)$, $\nabla R=-\frac{R\nabla \varphi}{\phi}$. Hence at this point,
 \be
 \begin{split}
 0\ge&\frac1n\varphi R^2-C_6|R|
 \end{split}
 \ee
 where $C_6$ depends only on $C_6, C_3$ and $r$. Hence
 $$
 \varphi^2 |R|\le n C_6.
 $$
on $B(2r)\times[0, T']$ and we conclude that $R\ge -C_7$ on
$B_{\hat{g}}(p, r)\times[0,T']$ for some $C_8$ depending only on $T', \hat g, r, \a(t)$.
 On the other hand,
 $$
 \frac{\p}{\p t}\log\lf(\frac{\det (h_{\a\bar\beta})(t)}{\det(h_{
\a\bar\beta}(0))}\ri)=-R\le C_7.
 $$
 So
 $$
\frac{ \det (h_{\a\bar\beta})(t)}{\det (\hat g_{\a\bar\beta}) }\le e^{C_7t}\frac{\det (h_{\a\bar\beta})(0)}{\det
(\hat g_{\a\bar\beta})} .
 $$
 on $B_{\hat{g}}(p, r)\times[0, T']$. Let $\lambda_i$ be eigenvalues of $h(t)$ with respect to $\hat{g}$. By part (i), $\lambda_i(x,T')\ge \a(T')$ for each $i$ and $x\in B_{\hat{g}}(p, r)$, and  the above inequality then implies $\lambda_i(x,T')\geq \beta(r, T')$ for some $\beta(r, T')$ depending only on the those constants listed in the Lemma.  Moreover, it is not hard to see that $\beta(r, T')$ can be chosen to depend continuously on $r, T'$ as $\a(t)$ is continuous.  The Lemma follows as $T'$ was chosen arbitrarily.
\end{proof}

\begin{rem} Given only a local solution $h(t)$ to \eqref{krf} on $B(1)\times[0, T)$ where $B(1)$ is the unit ball on $\C^n$, it is not hard to see from its proof that the conclusion of Lemma \ref{lemgboundedbelowboundedcurvature} will hold in $B(r)\times[0, T)$ for all $r\leq 1/2$.
\end{rem}

\section{\K Ricci flow: general existence Theorems}

We are now ready to state and prove our main existence Theorems for \eqref{krf} using the estimates in the previous section.  Theorems \ref{thmgboundedbelowbyboundedcurvature} and \ref{thmgequivalenttoboundedcurvature} provide general existence Theorems for \eqref{krf} when the initial \K metric is realized as a limit of a sequence of \K metrics satisfying certain properties.  The   curvature of the initial metric may be unbounded or even undefined.  As an application, we obtain an existence result for \K metric which is some perturbation of a complete \K metric with bounded curvature.   In Corollaries  \ref{existencetimeestimate-2} and \ref{existencetimeestimate} we apply the above Theorems to provide an estimate for the maximal existence time for \eqref{krf} assuming that the curvature of the initial metric is bounded. In some cases, we have long time existence.

In the following, we say that a sequence of smooth metrics $h_k$ converge smooth to a metric $g$ on a set  $U$, if $h_k$ converge to $g$ in $C^\infty$ norm on $U$.

 \begin{thm}\label{thmgboundedbelowbyboundedcurvature}
Let $g_0$ be a complete continuous Hermitian metric on a noncompact complex manifold $M^n$. Suppose there exists a sequence  $\{h_{k,0}\}$   of smooth complete \K metrics with bounded curvature on $M$
converging
uniformly on compact subsets to $g_0$ and another complete \K metric $\hg$ on $M$  with bounded curvature and holomorphic bisectional curvature bounded from above by $K$ such that

\begin{itemize}
\item[(i)]    $ h_{k,0} \ge \hat{g}$ for all $k$;
  \item [(ii)] for every $k$, \KR flow \eqref{krf} has
smooth  solution $h_k(t)$ with initial value $h_{k,0}$ on $M\times[0,T')$ for some $T'>0$
independent of $k$ such that the curvature of $h_k(t)$ is uniformly bounded on $M\times[0,T_1]$ for all $0<T_1<T'$;
\item [(iii)] The scalar curvature $R_k$ of $h_{k,0}$ satisfies: for any $r>0$,
there exists a constant $C_r>0$ such that
$R_k\geq -C_r$ on $B_{\hat{g}}(p, r)$ for some fixed point $p\in M$ and all
$k$.
\end{itemize}
Let $T=\min\{T',\frac1{2nK}\}$ if $K>0$, otherwise let $T=T'$.  Then the \KR flow \eqref{krf} has a complete smooth solution $g(t)$ on $M\times (0, T)$ which extends continuously to $M\times[0,T)$ with $g(0)=g_0$ and satisfies $g(t)\ge (1/n-2nK)\hat g$ on $M\times (0, T)$.

Moreover, if $g_0$ is smooth and $\{h_{k,0}\}$ converges smoothly and uniformly on
compact subsets of $M$, then $g(t)$ extends to a smooth solution to \eqref{krf} on $M\times
[0, T)$ with $g(0)=g_0$.

\end{thm}

\begin{proof}
By Lemma \ref{lemgequivalenttoboundedcurvature}, we have
\be\label{e-1}
h_k(t)\ge \lf(\frac1n-2Kt\ri)\hat g
\ee
where $K$ as long as $t<T_0=1/(2nK)$. By Theorem \ref{shishortime}, let $\hat g(t)$ be
the solution  \KR flow in the theorem with initial condition
$\hat g$. Then for any $1>\e>0$ small,    choose $0<t_0$ small enough so that $(1-\e)\hg(t_0)\le \hg\le (1+\e)\hg(t_0)$. Then we have
\be\label{e-1-1}
h_k(t)\ge \lf(\frac1n-2Kt\ri)(1-\e)\hat g(t_0)
\ee
and   $\hat g(t_0)$ has bounded geometry of infinite order.  By
Lemma \ref{lemgboundedbelowboundedcurvature}, for there is a positive
continuous function $\beta(r,t):[1, \infty)\times[0, T_0)\to \R$ such that for $r\ge1$
\be\label{e0}
h_k(t)\le\beta(r,t)\hat g(t_0).
\ee
 in $\hat{B}(p,r)\times[0,T)$
where $T=\min\{T',\frac1{2nK}\}$ and $p\in M$ is
a fixed point.  We conclude from Theorem \ref{t-EvanKrylov} (i), that passing to some
subsequence,
the $h_k(t)$'s converge to a solution $g(t)$ of \KR on
$M\times(0,T)$ so that \eqref{e-1} is true.   Moreover, if $g_0$ is smooth and
$\{h_k\}$ converges smoothly and uniformly to $g_0$ on compact sets, then we see from Theorem
\ref{t-EvanKrylov} (ii) that in fact $g(t)$ extends to a smooth solution on
$M\times [0,T)$ such that $g(0)=g_0$.

We now prove $g(t)$ converge uniformly on compact set to  $g_0$ as $t\to0$ when $g_0$ is only assumed to be continuous.
Fix any $x\in M$ and a local biholomorphism $\phi:B(1)\to M$
where $B(1)$ is the open unit ball in $\C^n$, and $\phi(0)=x$.  Consider the
pullbacks $\phi^*h_k(t)$, $\phi^*h_k=\phi^*h_k(0)$, $\phi^* \hat{g}$, which by
abuse of notation we will simply denote by $h_k(t)$, $h_k$, $\hat{g}$,
respectively, for the remainder of proof.  In particular, $h_k(t)$ solves \KR
flow \eqref{krf} on $B(1)\times [0, T)$.

Now by
our hypothesis on the convergence of $h_k$, given any $\delta>0$ we may find
$k_0$ such that $|h_{k_0,0}-g_0|_{\hat g}\le \delta$ and
\be
(1-\delta)h_{k_0,0}\le h_{k,0}\le (1+\delta)h_{k_0,0}
\ee
for all $k\ge k_0$.  On the other hand, by \eqref{e-1-1} and \eqref{e0} we can find $N>0$ such that
\be
N^{-1}h_{k_0, 0}\le h_k(t)\le Nh_{k_0, 0}
\ee
in $B(1)\times [0,T/2)$ for all $k\ge k_0$.
Then by Lemma \ref{lemgequivalenttoboundedcurvaturelocal}, there exists a continuous function $a(t)$ depending on $N, h_{k_0}$ and $\delta$ such that
$$
(1-a(t))\frac{(1-\delta)^2}{(1+\delta)}h_{k_0, 0}\le h_k(t)\le (1+a(t))(1+\delta)h_{k_0, 0}
$$
in $B(\frac12)\times[0, T/2)$ with $a(0)= n\sqrt{C(C-1)}$, with $C=(1+\delta)/(1-\delta)$. Note that $a(t)$ is independent of $k$. Letting $k\to\infty$ gives
\be\label{E2}
(1-a(t))\frac{(1-\delta)^2}{(1+\delta)}h_{k_0, 0}\le g(t)\le (1+a(t))(1+\delta)h_{k_0, 0}
\ee
in $B(\frac12)\times(0,T/2)$. We then get
\bee
\begin{split}
\limsup_{t\to0}&  |g(t)-g_0|_{\hat g}\\
\le &\limsup_{t\to0}  \lf( |g(t)-h_{k_0, 0}|_{\hat g}+|h_{k_0,0}-g_0|_{\hat g}\ri)\\
 \le&  \lf[  \lf| 1-(1-a(0))\frac{(1-\delta)^2}{(1+\delta)}  \ri|+ \lf|(1+a(0))(1+\delta) -1\ri|\ri]|h_{k_0,0}|_{\hg}\\
&+\delta|h_{k_0,0}|_{\hg}
\end{split}
\eee
uniformly on $B(\frac12)$.  Then letting $\delta \to 0$ above, and using the fact that $a(0)\to 0$ as $\delta\to0$, and \eqref{e-1-1} and \eqref{e0} we conclude that
$$
\limsup_{t\to0}  |g(t)-g_0|_{\hat g}=0.
$$
uniformly on $B(\frac12)$.  Hence $g(t)$ converge to $g_0$ uniformly on compact sets as $t\to0$.

\end{proof}

We do not have any bound on the curvature of solution $g(t)$ in the previous theorem. Also in the previous theorem, we assume that the \KR flow \eqref{krf} has solution with initial condition $h_{k,0}$ on a fixed time interval independent of $k$. We want to remove this assumption and obtain curvature bound for the solutions. In order to do this, we assume $h_{k,0}$ also has an uniform upper bound.

 \begin{thm}\label{thmgequivalenttoboundedcurvature}
 Let $g_0$ be a complete continuous Hermitian metric on a noncompact complex manifold $M^n$. Suppose there exists a sequence  $\{h_{k,0}\}$   of smooth complete \K metrics with bounded curvature on $M$
converging
uniformly on compact subsets to $g_0$ and another complete \K metric $\hg$ on $M$ with bounded curvature and holomorphic sectional curvature bounded from above by $K$  such that
 \begin{enumerate}

 \item[(i)] $C^{-1} \hat{g} \leq h_{k,0} \leq C \hat{g}$ for some $C$ independent
of $k$;
 \item[(ii)] $h_k$ has bounded curvature for every $k$.

 \end{enumerate}
Let $T=1/(2CnK)$ if $K>0$, otherwise let $T=\infty$. Then  the \KR flow \eqref{krf} has a smooth
solution $g(t)$ on $M\times (0, T)$ such that
\begin{enumerate}
\item [(a)] $(1/(nC)- 2Kt)\hat g\le g(t)\le B(t)\hat g$ on $M\times (0, T)$ for
some positive
continuous function $B(t)$ depending only on $C$, $\hat{g}$ and $n$.
\item [(b)] $g(t)$ has bounded curvature for $t>0$. In
particular, for any $0<T'<T$ and for any $l\geq 0$ there exists a constant $C_l$ depending only on $C$,
$l$, $T'$, $\hat{g}$ and the dimension $n$ such that
     $$
 sup_M |\nabla^l \Rm(g(t))|^2_{g(t)}\le \frac {C_l}{t^{l+2}},
 $$
\item [(c)] $g(t)$ converges
uniformly on compact subsets to $g_0$ as $t\to 0$.
\end{enumerate}

 Moreover, if $g_0$ is smooth and $\{h_{k,0}\}$ converges smoothly and uniformly on
compact subsets of $M$, then $g(t)$ extends to a smooth solution on $M\times
[0, T)$ with $g(0)=g_0$.
\end{thm}

\begin{proof}  For each $k$, let $h_k(t)$ be the solution to \eqref{krf} with initial condition $h_k$ from Theorem
\ref{shishortime} which is defined on $M\times[0,T_k)$ for some $T_k>0$. We
first claim that there is  such that $T_k\ge T$ for all $k$, where $T=1/(2nCK)$. By Lemma
\ref{lemgequivalenttoboundedcurvature}, there is a positive continuous
function $B(t):[0, T) \to \R$ independent of $k$ such that
$$
(1/n- 2nCKt)\hat g\le h_k(t)\le B(t)\hat g
$$
in $M\times[0,\min\{T_k,T\})$. As before, we may assume that $\hat g$ has bounded geometry of
infinite order. By Theorem \ref{t-EvanKrylov}, we conclude
that if $T_k<T$, then $|\Rm(h_k(t))|_{h_k(t)}$ are bounded in
$M\times[0,T_k)$.  By Theorem \ref{shishortime}, we see that one can extend
$h_k(t)$ so that $T_k\ge T$. Given upper and lower bounds on $h_k(t)$ as above, we may conclude from  Theorem \ref{t-EvanKrylov}, as in the proof of  Theorem
\ref{thmgboundedbelowbyboundedcurvature}, that there is a smooth
solution
to the \KR flow $g(t)$ on $M\times (0,T)$ satisfying condition (a) and (c)
from which we conclude, by Theorem \ref{t-EvanKrylov} (i), that condition (b)
is also satisfied.  Also, by Theorem \ref{thmgboundedbelowbyboundedcurvature},
we have that if $g_0$ is smooth and $\{h_{k,0}\}$ converges smoothly uniformly on
compact sets, then $g(t)$ extends to a smooth solution on $M\times [0,T)$
such that $g(0)=g_0$.

\end{proof}

\begin{cor}\label{perturbpotential} Let $(M^n, \hat g)$ be a complete \K manifold with bounded curvature. Suppose $u$ is real $C^2$ function on $M$ such that $|\hat \nabla u|+|u|=o(r)$ and
$$
A^{-1}\hat g\le \hat g+\ii\ddbar u\le A\hat g
$$
for some $A>1$, where $\hat\nabla$ is the covariant derivative with respect to $\hat g$.  Then for some $T>0$, the \KR flow \eqref{krf} has a complete solution $g(t)$ on $M\times [0, T)$ with $g(0)=g_0+\sqrt{-1}\partial\bar{\p} u$ and satisfying the conclusion of Theorem \ref{thmgequivalenttoboundedcurvature}.

\end{cor}

(Here $f=o(r^k)$ represents a positive function on $M$ such that $f(x)/d_p^k (x)$ approaches $0$ as $d_p(x)\to \infty$ where $d(x)$ is the distance function from some fixed point in $M$  relative to $\hat{g}$).

\begin{proof} Let $d(x)$ be the distance with respect to  $\hat g$ from $x$ to a fixed point $p\in M$.  As in the proof of Lemma \ref{lemgboundedbelowboundedcurvature}, we may assume without loss of generality that $d(x)$ is smooth with $|\hat\nabla d|, |\hat \nabla^2d|$ bounded on $M$.  Let $\phi$ be a smooth function on $\R$ such that $0\le \phi\le 1$, $\phi(s)=1$ for $s\le 1$ and $\phi(s)=0$ for $s\ge 2$, and $|\phi'|+|\phi''|\le c_1$ for some $c_1$. For any $k\ge 1$, let $\eta_k(x)=\phi(d(x)/k)$. Then $|\hat\nabla \eta_k|,
|\hat\nabla^2\eta_k|\le c_2/k$ on $M$ for some constant $c_2$ independent of $k$.  Now let $\{u_k\}$ be a sequence of smooth functions on $M$ which converging to $u$, uniformly on compact subsets of $M$  in the $C^2$ norm.  For each $k$ we have
\be\label{eq-example-1}
\begin{split}
 \ddbar (\eta_ku_j)=&\eta_k\ddbar u_j+u_j\ddbar \eta_k+\p u_j\wedge\bar\p \eta_k+\p\eta_k\wedge\bar\p u_j\\
 &\to \eta_k\ddbar u+u\ddbar \eta_k+\p u \wedge\bar\p \eta_k+\p\eta_k\wedge\bar\p u
\end{split}
\ee
uniformly on $M$ as $j\to\infty$. Since $\p\eta_k$ and $\p\bar\p \eta_k$ vanish on $B(k)$ and outside $B(2k)$,  and $|\hat \nabla u|+|u|=o(r)$,   for any $\e>0$ we have $|\p u_j\wedge \bar{\p}  \eta_k|_{\hat{g}}+|u_j\ddbar \eta_k|_{\hat{g}}\le \e$ if $j$ is large enough. Hence, for any $k$, we can find $u_{j_k}$ with $j_k\to\infty$ as $k\to\infty$ such that
$h_k=\hat g+\ii\ddbar (\eta_ku_{j_k})$ is a \K metric such that
$$
(2A)^{-1}\hat g\le h_k\le 2A\hat g.
$$

In particular,   $h_k$ is complete, outside a compact set $h_k=\hat g$ and thus has bounded curvature, and $h_k$ converges to $g_0$ uniformly on compact sets in $C^0$.  The corollary now follows from Theorem \ref{thmgequivalenttoboundedcurvature}

\end{proof}

\begin{rem} Note that if $|\hat{\nabla} u|=o(1)$, then $|u|=o(r)$. This will imply that $|\hat{\nabla}  u|+|u|=o(r)$.  Also from the proof of the theorem, if  $\hat g$ has bounded curvature and $u$ is a smooth function on $M$ such that $\p\bar\p u$, $u$ and $\hat{\nabla}  u$ are bounded, then for $\e>0$ small enough, the \KR flow with initial condition $\hat g+\e\ii\p\bar\p u$ has a short time solution.

\end{rem}

By Theorem \ref{thmgequivalenttoboundedcurvature}, one may obtain  estimates for maximal time interval of existence of the \KR flow constructed in  Theorem \ref{shishortime}.

\begin{cor}\label{existencetimeestimate-2}
 Let $M^n$ be a complex noncompact manifold and let $g_0$, $\hg$ be complete \K metrics with  bounded curvature on $M$. Suppose the holomorphic bisectional curvature of $\hg$ is  bounded above by $K$ and that $\hg\le g_0\le C\hg$ for some $C\ge1$.  Let $T=1/2nK$ if $K>0$, otherwise let $T=\infty$.  Then the \KR flow \eqref{krf} has a complete smooth solution $g(t)$ on $M\times[0,T)$ with $g(0)=g_0$ such that for all $t\in [0, T)$, $g(t)$ has bounded curvature and
 \be
  \lf(1/n-2Kt\ri)\hg\le g(t).
 \ee
     In particular, the \KR flow has a long time solution if the initial condition is a complete \K metric with non-positive and bounded holomorphic bisectional curvatures.
\end{cor}

\section{\K Ricci flow of U(n) invariant metrics on $\C^n$}\label{WZ}
In this section we apply Theorems
\ref{thmgboundedbelowbyboundedcurvature}, \ref{thmgequivalenttoboundedcurvature} to $U(n)$ invariant metrics on $\C^n$.

\subsection{Wu-Zheng's construction}
  We recall Wu-Zheng's construction in \cite{WZ} of smooth $U(n)$ invariant
metrics on $\C^n$.  Begin with a smooth function $\xi: [0, \infty)\to \R$ with
$\xi(0)=0$, and define functions $h, f : [0, \infty) \to \infty$ by

 \begin{equation}\label{hf}
 h(r):=Ce^{{\int^r_{0}-\frac{\xi(t)}{t}dt}}; \hspace{10pt}
f(r):=\frac{1}{r}\int^r_{0}h(t)dt
 \end{equation}
  where $h(0)=C>0$ and $f(0)=h(0)$.

Now define a $U(n)$ invariant metric $g$ on $\C^n$ by
 \begin{equation}\label{g}
 g_{i\bar{j}}=f(r)\delta_{ij}+f'(r)\overline{z_i}z_j.
\end{equation}
where $g_{i\bar{j}}$ are the components of $g$ in the standard coordinates
$z=(z_1,\dots,z_n)$ on $\mathbb{C}^n$  and $r=|z|^2$.  Notice that a different
choice of $h(0)$ simply corresponds to scaling the metric $g$ above. In the
following, we always take $C=1$, i.e. $h(0)=1$. Wu-Zheng \cite{WZ} proved

\begin{thm}\label{thmcurvature}[Wu-Zheng]

\begin{enumerate}
\item [1.] The metric $g$ above is complete iff
\begin{equation}\label{completeness}
f>0, \hspace{12pt} h>0, \hspace{12pt}
\int_0^{\infty}\frac{\sqrt{h}}{\sqrt{t}}dt=\infty \hspace{12pt}
\end{equation}
Conversely, up to scaling by a constant factor, every complete smooth $U(n)$
invariant \K metric on $\C^n$ can be generated in this way.
\item[2.]
At the point $z=(z_1,0,\dots,0)$, relative to the orthonormal frame $\{
e_1=\frac{1}{\sqrt{h}}\p_{z_1}, e_2=\frac{1}{\sqrt{f}}\p_{z_2},\dots,
\frac{1}{\sqrt{f}}\p_{z_n} \}$ with respect to $g_{i\jbar}$, we have
\begin{enumerate}
\item[(i)] $\displaystyle A=R_{1\bar{1}1\bar{1}}=\frac{\xi'}{h}$,
\item[(ii)] $\displaystyle B=R_{1\bar{1}i\bar{i}}=\frac{1}{(rf(r))^2}\int_0^r \xi'(t) \lf(\int_0^t
h(s)ds\ri) dt $,
\item[(iii)]
$\displaystyle C=R_{i\bar{i}i\bar{i}}=2R_{i\bar{i}j\bar{j}}=\frac{2}{(rf(r))^2}\int_0^r h(t)\xi(t)dt$,
\end{enumerate}
where $2\leq i\neq j\leq n$ and these are the only non-zero components of the
curvature tensor at $z$ except those obtained from $A, B$ or $C$ by the symmetric properties of $R$.
\end{enumerate}
\end{thm}

  In this section, $C$ always denotes the quantity in the above theorem.

By the above construction, Wu-Zheng \cite{WZ} proved the correspondence below
for positively curved metrics, while Yang later showed in \cite{Y} that this
extends to a correspondence for non-negatively curved metrics.

\begin{thm}\label{thmwz}[Wu-Zheng, Yang] There is a one to
one correspondence between the set of all smooth complete $U(n)$ invariant \K
metrics on $\C^n$ with non-negative holomorphic bisectional curvature (modulo
scaling by a constant factor)  and the set of all smooth function $\xi:[0,
\infty)\to \R$ satisfying
\begin{equation} \xi(0)=0, \hspace{12pt} \xi'\geq 0,
\hspace{12pt} \xi \leq 1\end{equation}
\end{thm}
\begin{rem}\label{rem-nonpositive}
One direction of the above correspondence is immediately obvious from Theorem \ref{thmcurvature} 2 (i).  In particular, it is obvious that if $g$ has non-negative (non-positive) holomorphic bisectional curvature then $\xi' \geq 0$ ($\xi'\leq 0$).

\end{rem}

\subsection{Applications of Theorems
\ref{thmgboundedbelowbyboundedcurvature} and \ref{thmgequivalenttoboundedcurvature} to $U(n)$ invariant metrics}
We now apply Theorems \ref{thmgboundedbelowbyboundedcurvature} and \ref{thmgequivalenttoboundedcurvature} to $U(n)$ invariant metrics. First we
have the following lemma.
\begin{lem}\label{lem-curvaturedecay} Let $g$ be a complete $U(n)$ invariant \K metric on $\C^n$ generated by $\xi$. \begin{enumerate}
                      \item [(i)] If $\lf|\frac{\xi'}{h}\ri|$ is uniformly bounded, then the curvature of $g$ is uniformly bounded.
                      \item [(ii)] If $\lim_{r\to\infty}\lf|\frac{\xi'(r)}{h(r)}\ri|=0$, and $\lim_{r\to\infty}rf(r)=\infty$ then the curvature of $g$ approaches to zero near infinity.
                    \end{enumerate}

\end{lem}
\begin{proof} (i) It is sufficient to prove that the holomorphic bisectional curvature is uniformly bounded under the assumption that $\lf|\frac{\xi'}{h}\ri| $ is uniformly bounded by $c$, say. By Theorem \ref{thmcurvature}, in the notations of the theorem it is sufficient to prove that $|A|, |B|, |C|$ are uniformly bounded. It is obviously $|A|\le c$. Now
\bee
\begin{split}
|B|\le &\frac1{r^2f^2}\int_0^r ch(t)dt\lf(\int_0^th(s) ds\ri)dt\\
\le&\frac c{r^2f^2}\lf(\int_0^r  h(t)dt\ri)^2\\
=&c
\end{split}
\eee
because $h>0$ and $rf(r)=\int_0^rh(t)dt$. Since
$$
|\xi(r)|\le \int_0^r|\xi'(t)|dt\le c\int_0^rh(t) dt,
$$
we have
$$
|C|\le 2c.
$$

(ii) If  $\lim_{r\to\infty}\lf|\frac{\xi'(r)}{h(r)}\ri|=0$, then $\lim_{r\to\infty}A=0$.
On the other hand, for any $\e>0$, there is $r_0$ such that $\lf|\frac{\xi'(r)}{h(r)}\ri|\le\e$ for $r\ge r_0$. Then
\bee
\begin{split}
|B|\le &\frac1{r^2f^2}\int_0^{r_0} |\xi'|(t) \lf(\int_0^th(s) ds\ri)dt+\e\\
\end{split}
\eee
Since $rf(r)\to\infty$ as $r\to\infty$, it is easy to see that $\lim_{r\to\infty}|B|=0$. Also
$$
|\xi|(r)\le \int_0^{r_0}|\xi'|(t)dt+\e\int_{r_0}^rh(t)dt
$$
if $r\ge r_0$. Hence
\bee
\begin{split}
|C|\le &\frac1{rf}\int_0^{r_0} |\xi'|(t) dt+\e,
\end{split}
\eee
and one can conclude that $\lim_{r\to\infty}|C|=0.$ From these (ii) follows.
\end{proof}
\begin{lem}\label{ctredecay} Let $\xi:[0,\infty)$ be a smooth function with
$\xi(0)=0$. Suppose $\xi(r)=a$ for some constant $a\le 1$ for all $r\ge r_0$.
Then $\xi$ generates a complete $U(n)$ invariant metric $g$ such that the
curvature of $g$ approaches $0$  as $x\to \infty$ on $\C^n$.
    \end{lem}
    \begin{proof} For $r\ge r_0$,
    $$
    \int_0^r\frac{\xi(t)}tdt= \int_0^{r_0}\frac{\xi(t)}tdt+a\log(\frac r{r_0}).
    $$
    Hence
    $
    h(r)=c_1r^{-a}
    $
    for some constant $c_1>0$ for all $r\ge r_0$. Since $a\le 1$, it is easy to
see that
    $$
    \int_0^\infty\frac{\sqrt h(r)}{\sqrt r}dr=\infty.
    $$
    Hence $g$ is complete by Theorem \ref{thmwz}. Also $\xi'=0$ near infinity, and

   $$
   rf(r)=\int_0^rh(t)dt\ge c_2+c_3\log r
   $$
   for some constants $c_2, c_3$ with $c_3>0$ because $a\le 1$. The result follows from Lemma \ref{lem-curvaturedecay}.
\end{proof}

\begin{thm}\label{mainthm-a}  Let $g_0$ be a smooth complete $U(n)$ invariant \K
metric on $\C^n$ generated by a smooth function  $\xi:[0, \infty) \to
\R$ with $\xi(0)=0$.  Suppose there exists $\hat{\xi}:[0, \infty) \to \R$ with $\hat \xi(0)=0$
 generating a smooth complete $U(n)$ invariant \K metric $\hat{g}$ with bounded curvature and holomorphic bisectional curvature bounded above by $K$, such that
for all $r\ge 0$
  $$
  \int_0^r\frac{\xi-\hat\xi}tdt\le c
  $$
  for some $c>0$ independent of $r$.   Let $T=1/(2nKe^{c})$ if $K>0$, otherwise let $T=\infty$.  Then the \KR flow \eqref{krf} has a smooth complete $U(n)$ invariant solution $g(t)$ on $M\times [0, T)$ with $g(0)=g_0$.

\end{thm}

\begin{proof}
As $\xi$ and $\hat{\xi}$ are smooth, for each $k\geq 0$ there exists a
$\delta_k >0$ and a smooth function ``cutoff" function $\eta_k:(-\infty,
\infty) \to \R$ satisfying
\begin{equation}
\eta_k(r) : \begin{cases} =1 &\mbox{if } -\infty<r\leq k\\ 0<\eta_k(r)<1
&\mbox{if } k< r< k+\delta_k\\
  =0 &\mbox{if }  k+\delta_k\leq r<\infty. \end{cases}
\end{equation}
and \begin{equation}\label{deltak} \int_k^{k+\delta_k}   \lf|\frac{  (\xi-\hat
\xi)}{t}\ri| dt  \leq 1/k\end{equation}
for all $k$.  Fix such a choice of $\eta_k's$, and consider the sequence of
functions $\{\xi_k\}:[0, \infty)\to \infty$ defined by
$$
\xi_k(r)=\eta_k\xi+(1-\eta_k)\hat{\xi}
$$
and let $h_k$ be the corresponding sequence of smooth $U(n)$ invariant \K
metrics.

Consider the sequence of $U(n)$ invariant metrics $h_k$ above.
\bee
\begin{split}
\int_0^r\frac{\xi_k(t)-\hat{\xi}(t)} {t}dt=&\int_0^r\frac{\eta_k(\xi-\hat
\xi)}tdt\\
= & \left\{
      \begin{array}{ll}
        \int_0^r\frac{\xi-\hat\xi}tdt, & \hbox{if $r\le k$;} \\
        \int_0^k\frac{\xi-\hat\xi}tdt+\a_k, & \hbox{if $r>k$  }
      \end{array}
    \right.
\end{split}
\eee
where
$$
|\a_k|\le \int_k^{k+\delta_k}\lf|\frac{\xi-\hat \xi}t\ri|dt\le \frac1k.
$$
where $C$ is the constant in the hypothesis.  Hence
$$
\int_0^r\frac{\xi_k(t)-\hat{\xi}(t)} {t}dt\le c+\frac1k.
$$
This implies, by \eqref{hf} and \eqref{g}, that
$$
\exp(-c-\frac 1k)\hat g\le h_k
$$
 In particular, $h_k$ is complete. Also, from \eqref{hf} and \eqref{g}, we have
$$
h_k\le c_k\hat g
$$
where
$c_k=\exp\lf(\int_0^{k+\delta_k}\lf|\frac{\xi-\hat \xi}t\ri|dt\ri).$  It is also easy to see that $h_k$ has bounded curvature for each $k$, and thus by Corollary \ref{existencetimeestimate-2}, for each $k$ there exists a solution $g_k(t)$ to \eqref{krf} on $M\times[0,  T_k)$ where $T_k=1/(2n\exp(-c-\frac 1k))$. By uniqueness \cite{CZ}, $g_k(t)$ is $U(n)$ invariant for all $t$. The result  now follows from Theorem \ref{thmgboundedbelowbyboundedcurvature}.

\end{proof}
  By Theorem \ref{mainthm-a}, we have
  \begin{cor}\label{maincor-a}
 Let $g_0$ be a smooth complete $U(n)$ invariant \K metric $g_0$ on $\C^n$
generated by a smooth function  $\xi:[0, \infty) \to \R$ with $\xi(0)=0$.
   If $\xi(r)\leq 1$,  then for some $T>0$ the \KR flow \eqref{krf} has a complete $U(n)$ invariant smooth solution $g(t)$ on $\C^n\times[0,T)$ with $g(0)=g_0$.   If in fact $\xi(r)\leq 0$, in particular if $\xi'\le0$, then the solution exists on $\C^n\times[0,\infty)$.
 \end{cor}
 \begin{proof} Let $\hat \xi$ be a smooth function on $[0,\infty)$ with $\hat \xi(0)=0$ and $\hat\xi(r)=1$ for $r\ge 1$. Then $\hat\xi$ generates a complete $U(n)$ invariant \K metric with bounded curvature by
by Lemma \ref{ctredecay}. The first result follows from Theorem
\ref{mainthm-a}.

 If $\xi\le0$, then we can choose $\hat\xi=0$ which generates the standard Euclidean metric. The second result also follows from  Theorem
\ref{mainthm-a}.
\end{proof}

We do not have any curvature bound on the solution in Theorem
\ref{mainthm-a}. In the next theorem, the solution also has some curvature bound.

  \begin{thm}\label{mainthm}  Let $g_0$ be a smooth complete $U(n)$ invariant \K
metric on $\C^n$ generated by a smooth function  $\xi:[0, \infty) \to
\R$ with $\xi(0)=0$. Suppose there exist $\a\le 0$  and $\beta$ such that for all $0<a<r$,
\be\label{eq-main-1}
\int_a^r \frac{(\a-\xi)}{t} dt, \hspace{12pt} \int_a^r \frac{(\xi-1)}{t} dt
\leq \beta.
\ee
Then for some $T>0$ the \KR flow \eqref{krf} has a complete smooth $U(n)$ invariant
solution $g(t)$ on $\C^n\times [0, T)$ with $g(0)=g_0$. Moreover, for every $\l \geq 0$ there exists a
constant $c_l$ depending only on  such that
\be\label{curvatureestimates}
\sup_{p\in\C^n} \|\nabla^l \Rm(p, t)\|^2_t \leq \frac{c_l}{t^{l+2}}
\ee
on $\C^n\times(0, T)$.

If in addition,
\be\label{eq-main-2}\int_0^r \frac{\xi}{t}<\sigma
 \ee
for some constant $\sigma$ independent of $r$, then the above solution to \KR flow is defined on $\C^n\times [0,\infty)$ and satisfies \eqref{curvatureestimates} on $\C^n\times(0, T')$ for some $T'>0$.
\end{thm}

 \begin{cor}\label{maincor}
 Let $g_0$ be a smooth complete $U(n)$ invariant \K metric $g_0$ on $\C^n$
generated by a smooth function  $\xi:[0, \infty) \to \R$ with $\xi(0)=0$.
  If $\a\leq \xi(r)\leq 1$ for some $\a\le 0$, in particular if $\xi'\ge 0$ so that $g_0$ has nonnegative holomorphic bisectional curvature,  then for some $T>0$ the \KR flow \eqref{krf} has a smooth
solution on $\C^n\times [0, T)$ with $g(0)=g_0$ and satisfies \eqref{curvatureestimates}. Moreover, the solution $g(t)$ has nonnegative holomorphic bisectional curvature for $t\in(0,T)$.

  If in fact $c\leq \xi \leq 0$ for all $r$, then the solution exists for all time and satisfies \eqref{curvatureestimates} on $\C^n\times(0,T')$ for some $T'$ depending only on $c$ and $n$.
  \end{cor}
 \begin{proof}
 If $c\leq \xi(r)\leq 1$ (or $c\leq \xi(r)\leq 0$) for some $c$, then the conditions of   Theorem \ref{mainthm}  clearly hold.  In case $g_0$ has non-negative holomorphic bisectional curvature, the fact that $g(t)$ has non-negative bisectional curvature for all $t\in [0, T)$ provided $g_0$ was proved in \cite{YZ}.
 \end{proof}

\begin{rem} Let $\hat \xi:[0, \infty)\to \R$ to be smooth with $\hat \xi(0)=0$ and $\hat \xi(r)=1+1/\ln r$ for $r\geq 1$ say.  Then from the proof of Proposition \ref {propxiequivalenttoboundedcurvature}, it is not hard to see that the corresponding $\hat{g}$ is complete with bounded curvature.  Now it is easy to construct a smooth function $\xi \geq \hat{\xi}$ satisfying the assumptions in Theorem \ref{mainthm-a}, where the corresponding $g$ is complete with unbounded curvature.  Thus $\xi$ satisfies the assumptions in Theorem \ref{mainthm-a}, while it is also easy to see that $\xi$ does not satisfy the assumptions in Theorem \ref{mainthm}.
\end{rem}

\subsection{Proof of Theorem \ref{mainthm}}

By Theorem \ref{thmgequivalenttoboundedcurvature}, Theorem \ref{mainthm}
will follow once we produce a sequence $\xi_k$ and function $\hat{\xi}$ such
that the corresponding $U(n)$ invariant \K metrics $h_k$, $g_0$ and $\hat{g}$
satisfy the hypothesis of Theorem \ref{thmgequivalenttoboundedcurvature}.  We
begin by proving the existence of such a function $\hat{\xi}$:

\begin{prop}\label{propxiequivalenttoboundedcurvature}
Under assumptions \eqref{eq-main-1} of Theorem  \ref{mainthm} on $\xi$, there exists $\hat{\xi}$ such
that the corresponding $U(n)$ invariant metric $\hat{g}$ has bounded curvature   and
\begin{equation}\label{equivalence2}
c^{-1}\hat{g} \leq g_0\leq c\hat{g}
\end{equation}
on $\C^n$ for some constant $c>0$. If in addition, \eqref{eq-main-2} is true, then $\hat\xi$ can be chosen to be  nonpositive.

\end{prop}

\begin{proof} Assume \eqref{eq-main-1} is true. We consider three different cases.

{\it Case 1}: Suppose there is $c'>0$ such that $\int_1^r\frac{\xi-1}t dt\ge c'$ for all $r\ge1$. Let $\hat\xi$ be a fixed smooth function on $[0,\infty)$ such that $\hat\xi(0)=0$ and $\hat\xi(r)=1$ for $r\ge 1$. Let $\hat g$ be the complete $U(n)$ invariant metric generated by $\hat \xi$. Then there is $c''$ for any $1\ge r>0$,
$$
\lf|\int_{0}^r\frac{\xi-\hat\xi}tdt\ri|\le  \lf|\int_0^r\frac{\xi}tdt\ri|+\lf|\int_0^r\frac{\hat\xi}tdt\ri|\le c''
$$
for some $c''$. For $r\ge 1$
$$
\int_1^r\frac{\xi-\hat\xi}tdt=\int_1^r\frac{\xi-1}tdt\le \beta
$$
and
$$
\int_1^r\frac{\xi-\hat\xi}tdt=\int_1^r\frac{\xi-1}tdt\ge c'
$$
Hence the $U(n)$ invariant \K metric $\hg$ generated by $\hat\xi$ satisfies the required conditions in the Proposition.\vskip.2cm

{\it Case 2}: Suppose there is $c'>0$ such that $\int_1^r\frac{\a-\xi}t dt\ge c'$ for $r\ge1$.
Let $\hat\xi$ be a fixed smooth function on $[0,\infty)$ such that $\hat\xi(0)=0$ and $\hat\xi(r)=\a$ for $r\ge 1$. Then as in the previous case,   the $U(n)$ invariant \K metric $\hg$ generated by $\hat\xi$ satisfies the required conditions in the Proposition. Note that in this case, $\hat\xi$ can be chosen to be nonpositive.\vskip.2cm

{\it Case 3}: Suppose as a function of $r$, $\int_1^r\frac{\xi-1}tdt$  is not bounded from below and $\int_1^r\frac{\xi-\a}tdt$ is not bounded from above. We want to find $\hat \xi$ and $1\le a_0<a_1<a_2\dots\to\infty$ $\hat\xi$ generates a complete $U(n)$ metric $\hg$ such that
\be\label{eq-construction-1}
\int_{a_{2i}}^{a_{2(i+1)}}\frac{\xi-\hat\xi}tdt=0
\ee
 for all $i\ge0$;
\be\label{eq-construction-2}
    \lf|\int_{a_{2i}}^r\frac{\xi-\hat\xi}tdt\ri|\le   c_1
\ee
 for some $c_1$ for all $i\ge0$ and for all $r\in[a_{2i},a_{2(i+1)})$; and
\be\label{eq-construction-3}
        \lf|\frac{ \hat\xi'(r) }{\hat h(r)}\ri|\le    c_2
\ee
 for some $c_2$ for all $r\ge0$. Then by Lemma \ref{lem-curvaturedecay}, we can conclude that $\hg$ satisfies the conditions of the Proposition.

Fix  a smooth function  $\rho$   on $\R$, such that
\bee
\rho(t)=\left\{
          \begin{array}{ll}
            1, & \hbox{if $t\le 1+\e$;} \\
            \a, & \hbox{if $t\ge 3-\e$, }
          \end{array}
        \right.
\eee
and $\rho'\le0$, where $\e>0$ is small enough so that $1+\e<3-\e$. Then $\a\le\rho\le 1$.

Let $\hat\xi$ be a smooth function on $[0,1]$ with $\hat\xi(0)=0$ and $\hat\xi(r)=1$ near $r=1$, $\a\le \hat\xi\le 1$. We are going to find $a_i$ inductively. Let $a_0=1$.
$$
\int_{a_0}^{3a_0}\frac{\xi-\rho}tdt=\int_{a_0}^{3a_0}\frac{\xi(t)-1+1-\rho(t)
} tdt\le \beta+(1-\a)\log 3.
$$
Since $\int_{3a_0}^r\frac{\xi-\a}tdt$ is not bounded from above, there is a {\it first} $a_1>3a_0$ such that
$$
\int_{a_0}^{3a_0}\frac{\xi-\rho}tdt+\int_{3a_0}^{a_1}\frac{\xi-\a}tdt=c_3
$$
where $c_3=\beta+(1-\a)\log 3+1$. On the other hand,

$$
\int_{a_1}^{3a_1}\frac{\xi-(1+\a-\rho(\frac{t}{a_1}))}tdt \ge -\beta- (1-\a)\log 3.
$$
Since $\int_{3a_1}^r\frac{\xi-1}tdt$ is not bounded from below, there exists  a {\it first} $a_2>3a_1$, such that
$$
\int_{a_1}^{3a_1}\frac{\xi-(1+\a-\rho(\frac{t}{a_1}))}tdt+\int_{3a_1}^{a_2}\frac{\xi-1}tdt=-c_3
$$
Define
\bee
\hat\xi(r)=\left\{
  \begin{array}{ll}
    \rho(r), & \hbox{if $a_0\le r\le 3a_0$;} \\
   \a, & \hbox{if $3a_0<r\le a_1$;}\\
1+\a-\rho(\frac{r}{a_1} ),& \hbox{if $a_1<r\le 3a_1$;}\\
 1, & \hbox{if $3a_1<r\le a_2$.}\\
  \end{array}
\right.
\eee
It is easy to see that $\hat\xi$ is smooth on $[0,a_2]$ with $\hat\xi(r)=1$ near $a_2$. Moreover, $\a\le\hat\xi\le 1$, and
\bee
\int_{a_0}^{a_2}\frac{\xi-\hat\xi}tdt=0.
\eee
so \eqref{eq-construction-1} is true for $i=0$. It is easy to see that
\bee
|\xi'|\le \frac {c_4}r
\eee
where $c_4= 3\max|\rho'|$.

For $a_0\le r\le a_1$, by the definition of $a_1$ we have
\bee
\int_{a_0}^{r}\frac{\xi-\hat\xi}tdt\le c_3.
\eee
For $a_1<r\le a_2$,
\bee
\begin{split}
\int_{a_0}^{r}\frac{\xi-\hat\xi}tdt=&
\lf(\int_{a_0}^{a_1}+\int_{a_1}^r\ri)\frac{\xi-\hat\xi}tdt\\
\le& c_3+\int_{a_1}^r\frac{\xi-1+1-\hat\xi}tdt\\
\le &c_3+\beta+(1-\a)\log 3.\\
\end{split}
\eee
Hence for $a_0\le r\le a_2$,
$$
\int_{a_0}^{r}\frac{\xi-\hat\xi}tdt\le 2c_3.
$$
Similarly, one can prove that
$$
\int_{a_0}^{r}\frac{\xi-\hat\xi}tdt\ge -2c_3.
$$
To summarize, we have find $\hat\xi(r)$ and $a_0<a_1<a_2$ such that $\hat\xi$ is smooth and defined on $[0,a_2]$ with  $\a\le \hat\xi\le1$, satisfying \eqref{eq-construction-1} with $i=0$, \eqref{eq-construction-2} with $i=0$, $c_1=2c_3$, and $|\hat\xi'|\le\frac{c_4}r$ on $[a_0,a_2]$. Moreover, $\hat\xi(r)=1$ near $r=a_2$.

From the above construction, it is easy to see that one can continue and find $a_2<a_3<a_4\dots\to\infty$ and $\hat\xi$ with $\a\le\hat\xi\le 1$ satisfying \eqref{eq-construction-1} with and   \eqref{eq-construction-2} with $c_1=2c_3$, and $|\hat\xi'|\le\frac{c_4}r$ on $[a_0,\infty)$.

Since $\hat\xi\le 1$,
$$
\hat h(r)\ge c_5\exp(-\int_1^r\frac1tdt)\ge \frac{c_5}r
$$
for some $c_5>0$ for all $r\ge 1$. Combing with the fact that $|\hat\xi'|\le\frac{c_4}r$ on $[a_0,\infty)$, we conclude that \eqref{eq-construction-3} is also true.

Suppose in addition $\xi$ satisfies \eqref{eq-main-2}. If $\int_0^r\frac{\xi}tdt$ is uniformly bounded from below, then one can take $\hat\xi\equiv0$.  If $\int_1^r\frac{\xi}tdt$ is not bounded from below
and $\int_1^r\frac{\xi-\a}tdt$ is not bounded from above, then one can proceed as in the proof of Case 3 in the above, by taking $\rho=0$ near $r=1$ instead. Then one can get $\hat\xi$ to be nonpositive.
This completes the proof of the Proposition.
\end{proof}

Now we are ready to prove Theorem \ref{mainthm}.
\begin{proof}[Proof of Theorem \ref{mainthm}] Let $\hat g$ be the $U(n)$
invariant \K metric with bounded curvature generated by $\hat \xi$ defined in
Proposition \ref{propxiequivalenttoboundedcurvature}, so that
\be\label{eq-long-1}
c_1^{-1}\hat g\le g_0\le c_1\hat g
\ee
for some $c_1>0$ as in Proposition \ref{propxiequivalenttoboundedcurvature}. As in the proof of Theorem \ref{mainthm-a}, choose   $\delta_k
> 0$ and  smooth function ``cutoff" functions $\eta_k:(-\infty, \infty) \to \R$
satisfying
\begin{equation}
\eta_k(r) : \begin{cases} =1 &\mbox{if } -\infty<r\leq k\\ 0<\eta_k(r)<1
&\mbox{if } k< r< k+\delta_k\\
  =0 &\mbox{if }  k+\delta_k\leq r<\infty. \end{cases}
\end{equation}
and \begin{equation}\label{deltak} \int_k^{k+\delta_k}   \lf|\frac{  (\xi-\hat
\xi)}{t}\ri| dt  \leq 1\end{equation}
for all $k$.  Let $\{\xi_k\}:[0, \infty)\to \infty$ be defined by
$$
\xi_k(r)=\eta_k\xi+(1-\eta_k)\hat{\xi}.
$$
Then as in the proof of Theorem \ref{mainthm-a}, each $\xi_k$ generates a $U(n)$ invariant \K metric $h_k$ so that

\be\label{eq-long-2}
c_2^{-1}\hat g\le h_k\le c_2\hat g
\ee
for some constant $c_2>0$, for all $k$. Now recall that the curvature of $\hat g$ is bounded by a constant $K$ as in
Proposition \ref{propxiequivalenttoboundedcurvature}, and thus by Theorem \ref{shishortime} we may assume without loss of generality that $\hat g$ has bounded geometry of order infinity.  In particular, the formula of curvature in Theorem
\ref{WZ} implies that each $h_k$ also
has bounded curvature.  We also clearly have $h_k \to g_0$ uniformly and smoothly on compact subsets of $M$.
By Theorem \ref{thmgequivalenttoboundedcurvature}, there is a solution $g(t)$
of the \KR flow with initial condition $g_0$ on $M\times[0,T)$ for some $T>0$ so
that
$$
|| \Rm(g(t))||^2_{g(t)}\le \frac{c_3}t
$$
for some $c_3>0$ and for all $0<t<T$. The estimates for $||\nabla^l\Rm||$ for each $l\geq 0$ then follows
from the general results of \cite{Shi1}.

If in addition that
$$
\int_0^\infty\frac{\xi}tdt<\sigma
$$
for some $\sigma$ for all $r$, then $\hat\xi$ can be chosen to be nonpositive. Then $\xi_k$ is nonpositive near infinity. Therefore the \KR flow with initial condition $h_k$ has longtime solution $h_k(t)$ by Theorem \ref{mainthm-a}. On the other hand,  $g_0\ge c_4 g_e$ for some constant $c_4>0$, where $g_e$ is the Euclidean metric on $\C^n$. By \eqref{eq-long-1} and \eqref{eq-long-2},
$$
h_k\ge c_5 g_e
$$
for some constant $c_5>0$ for all $k$. By Theorem \ref{thmgboundedbelowbyboundedcurvature},    that there exists a longtime solution $g(t)$ to \eqref{krf} with initial condition $g_0$. On $(0,T)$, $g(t)$ is the same as before. Hence  $g(t)$ also satisfies \eqref{curvatureestimates} on $\C^n\times(0, T)$.
\end{proof}

The long time existence results in Theorem \ref{mainthm} are basically for $U(n)$ invariant metrics with non-positive curvature.  The following Theorem gives a longtime existence result for $U(n)$ invariant metrics with non-negative curvature.
\begin{thm}\label{longtimenonnegativecurvature} Let $g_0$ be a smooth complete $U(n)$ invariant \K metric on $\C^n$
generated by a smooth function $\xi:[0, \infty) \to \R$ with $\xi(0)=0$.  Suppose $\xi(r)=a$ for $r$ sufficiently large where $a\le 1$.  Then the \KR flow has a smooth complete $U(n)$ invariant solution $g(t)$ on $\C^n\times [0, \infty)$ with $g(t)=g_0$. In general, if there is $C>0$ such that
\be\label{eq-Un-longtime-1}
-C\le \int_1^r\frac{\xi-a}tdt\le C
\ee
for some $a\le 1$ for all $r>1$ and such that $|\xi'|=o(r^{-a})$,  then the \KR flow has a smooth complete $U(n)$ invariant solution $g(t)$ on $\C^n\times [0, \infty)$ with $g(t)=g_0$ such that the curvature of $g(t)$ is uniformly bounded on $M\times[0,T]$ for all $T<\infty$.

\end{thm}
\begin{rem} If $a\le0$, then we have long time solution by Theorem \ref{mainthm}. However, there is no curvature bound obtained for all $t$ in that theorem. In that theorem, we can only conclude that the curvature of the solution is uniformly bounded in $M\times[0,T]$ for some $T>0$.
\end{rem}
\begin{proof}
 Suppose \eqref{eq-Un-longtime-1} is true.  We want to prove that the curvature of $g$ tends to zero as $x\to\infty$. Consider the case that $ a<1$, then
$$
h\ge c_1 r^{-a}
$$
for large $r$ for some $c_1>0$. Hence $rf\ge c_2r^{1-a}$ for $r$ large for some $c_2>0$ and $rf(r)\to\infty$ as $r\to\infty$. $|\xi'|=o(r^{-a})$ implies
$|\frac{\xi'}{h}|=o(1).
$
By Lemma \ref{lem-curvaturedecay}, the curvature of $g_0$ approaches to zero at infinity.

Suppose $a=1$, then there is $c_3>0$ such that
$$
h\ge \frac{c_3}r
$$
for $r$ large. So
$$
rf\ge c_4\log r
$$
for some $c_4>0$ if $r$ is large. By Lemma \ref{lem-curvaturedecay}, the curvature of $g_0$ also approaches to zero at infinity.

 The Theorem now follows from the above curvature decay estimates, Lemma \ref{injectivity} below which implies the injectivity radius of $g$ is bounded below on $\C^n$, and Theorem \ref{longtimeexistence} because $\C^n$ has a strictly pluri-subharmonic function.
\end{proof}

\begin{lem}\label{injectivity} Let $\xi(r)=a$ for all $r$ sufficiently large and $a\leq 1$.  Let $g$ be the corresponding $U(n)$ invariant \K metric on $\C^n$. Then  the injectivity radius of $g$ is bounded below by a positive
constant on $\C^n$
\end{lem}
\begin{proof}
  We begin by assuming $a<1$.  Indeed, this will be sufficient for our applications.  By the estimate in \cite{CGT} and by the fact that the curvature of $g$ is
bounded by Lemma \ref{ctredecay}, in order to prove the injectivity radius of
$g$ is positive on $\C^n$ it is sufficient to prove
 there is a constant $c>0$ such that
 $$
 V_g(B_g(x,1))\ge c
 $$
 for all $x$ where $B_g(x,1)$ is the geodesic ball of radius 1 with center at
$x$ with respect to $g$. Let $\tau$ be the geodesic distance from the origin,
then for $a<1$ and   $r=|z|^2>r_0$.
\be\label{eq-distance-1}
\tau(z)=\int_0^r\frac{\sqrt h}{2\sqrt s}ds=c_1+c_2r^{\frac12(1-a)}
\ee
for some constants $c_1$, $c_2$ with $c_2>0$.   So
$$
rf(r)=c_3+c_4(\tau-c_1)^2
$$
with $c_4>0$.
$$
V(B_g(0;\tau))=c_n(rf)^n=\lf(c_3+c_4(\tau-c_1)^2\ri)^n.
$$
where $\tau=\tau(r)$ is given by \eqref{eq-distance-1}.
Hence if $\tau$ is large, then
\be\label{eq-vol-1}
V_g(B_g(0;\tau+1)\setminus B_g(0;\tau-1))\ge c_5\tau^{2n-1}
\ee
for some $c_5>0$ independent of $\tau$. Let $\mathcal{F}$ be a maximal disjoint
family of $B_g(x,1)$ with $x\in \p B_g(0,\tau)$. Let $\mathcal{C}=\{x|\
B_g(x,1)\in  \mathcal{F}$ and let $N=N(\tau)=\#(\mathcal{C})$.
We claim that $\bigcup_{x\in \mathcal{C}}B_g(x,3)\supset B_g(0;\tau+1)\setminus
B_g(0;\tau-1)$. In fact, if $y\in B_g(0;\tau+1)\setminus B_g(0;\tau-1)$, then
there is $y'\in \p B_g(0;\tau)$ such that $d_g(y,y')<1$. On the other hand,
there is $ x\in \mathcal{C}$ with $d_g(x,y')<2$. From these the claim
followers.

Since $g$ is $U(n)$ invariant, $v=v(\tau)=V_g(B_g(x,3))$ is constant for $x\in
\p B_g(0,\tau)$. Hence we have
\bee
Nv\ge c_5\tau^{2n-1}
\eee
and
\be\label{eq-vol-2}
v\ge \frac{c_5}N\tau^{2n-1}.
\ee
By the expressions of $h$ and $f$, on $\p B_g(0,\tau)=\p B_0(0,\sqrt r)$,
$c_6^{-1}r^{-a}g_0\le g\le c_6r^{-a}g_0$  for some $c_6>0$ if $r$ is large,
where $B_0(0,\sqrt r)$ is the Euclidean ball with radius $\sqrt r$ and center
at the origin. Let $B_g^\tau(x,\rho)$ be the geodesic ball with respect to the
intrinsic distance of $\p B_g(0,\tau)$. Define $B_0^\tau(x,\rho)$ similarly
with respect to $g_0$.

Since $B_g(x,1)\supset B_g^\tau(x,1)$.
and $B_0^\tau(x,c_6^{-1}r^{\frac  a2})\subset B_g^\tau(x,1)$. Hence
$\{B_0^\tau(x,c_6^{-1}r^{\frac  a2})| x\in \mathcal{C}\}$ is a disjoint family.
Hence
$$
NV_{g_0}(B_0^\tau(x,c_6^{-1}r^{\frac  a2}))\le V_{g_0}(\p B_0(0,\sqrt
r))=c_nr^{\frac{2n-1}2}.
$$
where $c_n$ is the volume of the unit sphere in $\C^n$. Let $\rho=r^\frac12$,
then the volume of the geodesic ball of radius $s_0$ in $\p B_0(0,\rho)$ is
$$
c_n\rho^{n-2}\int_0^{s_0}\sin^{2n-2}\frac s\rho ds.
$$
where $c_n$ is a positive constant depending on $n$. Let $s_0=c_6^{-1}r^{\frac
a2})$. Then $s_0/\rho\to0$ as $r\to\infty$.  Hence for $r$ large,
 \bee
 \begin{split}
 V_{g_0}(B_0^\tau(x,c_6^{-1}r^{\frac  a2}))\ge &c_7\int_0^{s_0}  s ^{2n-2} ds\\
 =&c_8  s_0^{2n-1}.
 \end{split}
 \eee
Hence
\bee
\begin{split}
v\ge& c_5N^{-1}\tau^{2n-1}\\
\ge& c_n^{-1}c_5c_8\tau^{2n-1} r^{-\frac{2n-1}2}   s_0^{2n-1}\\
\ge& c_9
\end{split}
\eee
 for some positive constant $c_9$ independent of $\tau$.

 We now consider the case when $a=1$.  Consider Cao's cigar soliton $\tilde{g}$ which is a complete $U(n)$ invariant \K metric on $\C^n$ .  It is shown in \cite{WZ} $\tilde{g}$ has positive sectional curvatures and is generated by $\txi$ satisfying

\be\label{eq-cigar-1}
\int_0^\infty\frac{\xi-\txi}t dt<\infty
\ee
since $\xi(r)=1$ for sufficiently large $r$ (see Theorem 3 in  \cite{WZ}).

In particular, by \eqref{hf} and \eqref{g} it follows that $g$ and $\tilde{g}$ are uniformly equivalent and thus $$V_{g}(B_g(p, 1))\geq C V_{\tilde{g}}(B_{\tilde g}(p, 1))$$ for some $C>0$
for all $p\in \C^n$ and for some constant $C$ independent of $p$.  To bound the injectivity radius of $g$ from below, it suffices to prove that the volume in the RHS above is uniformly bounded below.  This follows from \cite{GM} since $\tilde{g}$ is complete, and has bounded positive sectional curvatures.  For completeness, we include a proof below that $\tilde{g}$ has bounded curvature.  By Wu-Zheng:

Let $\tilde{\phi}=r\tilde{f}$ and $t=\log r$.   $\tilde{\phi}'=r\tilde{h}$. Hence $\tilde{\phi}>0$, $\tilde{\phi}'>0$ for $t>-\infty$. Here all primes on $\tilde{\phi}$ are with respect to $t$. Since $A, B, C>0$, we only need to prove that $A, B, C$ are bounded from above. It is sufficient to prove that $A, B, C$ are bounded from above for $t\ge0$. For $t\ge 0$, by \cite[\S4]{WZ}

$$
A=n(1+\frac{n-1}\tphi)-\tphi'\lf(1+\frac{2(n-1)}\tphi+\frac{n(n-1)}{\tphi^2}\ri)\le n(1+\frac{n-1}{\tphi(0)}),
$$
because $\tphi'>0$, $\tphi>0$. So $A$ is bounded.
$$
B=\frac{1}{(r\tilde{f})^2}\int_0^r\frac{d\txi}{dr}\lf(\int_0^t \tilde{h}(s)ds\ri)dt \le \frac{1}{r\tilde{f}}
$$
because $\frac{d\txi}{dr}>0$, $\txi(r)\le 1$. On the other hand by \eqref{eq-cigar-1}, $\tilde{h}(r)\ge cr^{-1}  $ for $r\ge 1$. Hence $r\tilde{f}\sim c\log r$. So $B$ is bounded. Similarly, $C$ is also bounded.

 \end{proof}
 \begin{rem}\label{rem-injectivity} In case $1>a\ge0$, we may simply compare $g$ with a metric $\tilde g$ with nonnegative bisectional curvature generated by $\tilde\xi$ with $\tilde \xi=a$ near infinity. In this case, $\tilde g$ has maximum volume growth by \cite{WZ}. Hence each geodesic ball of radius 1 is bounded below by a constant which is uniform for all points. So this is also true for $g$.
 \end{rem}


\begin{thebibliography}{1000}
\bibitem[CW]{CW}E.  Cabezas-Rivas; B., Wilking, {\sl How to produce a Ricci Flow via Cheeger-Gromoll exhaustion} arXiv:1107.0606 (2011).

\bibitem[C]{C} H.D. Cao,
  {\sl Deformation of K\"ahler metrics to K\"ahler
Einstein metrics on compact Kahler manifolds},
  Invent. Math. \textbf{81} (1985),  359--372.
\bibitem[CT]{CT} A. Chau; L.-F. Tam, {\sl On a modified parabolic complex Monge-\A equation with applications},Math. Z. \textbf{269} (2011), no. 3-4, 777--800.

\bibitem[CGT]{CGT} J. Cheeger; M. Gromov;  M. Taylor, {\sl Finite propagation speed, kernel estimates for functions of the Laplace operator, and the geometry of complete Riemannian manifolds}, J. Diff. Geom. \textbf{17}(1) (1982), 15--53.






 \bibitem[CZ]{CZ} Chen, B.-L.; Zhu, X.-P., {\sl Uniqueness of the Ricci flow on complete noncompact manifolds} J. Differential Geom.\textbf{ 74} (2006), no. 1, 119--154.

\bibitem[Ev]{Ev} L.C. Evans, {\sl Classical solutions of fully nonlinear, convex, second order elliptic equations}, Comm. Pure Appl. Math, \textbf{35} (1982), 333--363.

\bibitem[GM]{GM} D. Gromoll; W. Meyer, {\sl On complete open manifolds of positive curvature}, Ann. of Math. (2)  \textbf{90}  (1969), 75--90.

\bibitem[GT]{GT} G. Giesen; P.M. Topping, {\sl Existence of Ricci flows of incomplete surfaces}, Comm. Partial Differential Equations \textbf{36} (2011), no. 10, 1860--1880.



\bibitem[Kr]{Kr} N.V. Krylov, {\sl Boundedly nonhomogeneous elliptic and parabolic equations}, Izvestia Akad. Nauk. SSSR \textbf{46} (1982), 487--523. English translation in Math. USSR Izv. \textbf{20}
(1983), no. 3, 459--492.

\bibitem[KL]{KL} H. Koch and T. Lamm, Geometric flows with rough initial data. Asian J. Math.,
16(2):209–235, 2012.

\bibitem[NT]{NT} Ni, L.; Tam, L.-F., {\sl \KR flow and the Poincar\'e-Lelong equation}, Comm. Anal. Geom. \textbf{12} (2004), no. 1-2, 111--141.

\bibitem[SW]{SW} M. Sherman; B. Weinkove, {\sl Interior derivative estimates for the \KR flow},  Pacific Journal of Mathematics, \textbf{257}(2) (2012), 491--501.
\bibitem[S1]{Shi1} W.,-X. Shi, {\sl Deforming the metric on complete Riemannian manifolds}, J. Differential
Geom. 30 (1989), no. 1, 223--301.

    \bibitem[S2]{Shi2} W.,-X., Shi, {\sl Ricci Flow and the uniformization on complete non compact \K
manifolds}, J. of Differential Geometry \textbf{45} (1997), 94--220.

\bibitem[Si]{Si} M. Simon, {\sl Deformation of $C^0$ Riemannian metrics in the direction of their Ricci curvature},
Comm. Anal. Geom. \textbf{10}(2002), no. 5, 1033--1074.


\bibitem[ST]{ST}Song, J, and Tian, G., {\sl The K¨ahler-Ricci flow on surfaces of positive Kodaira dimension}, Invent. Math. 170 (2007), no. 3, 609–653.

\bibitem[SSS1]{SS1}O. C. Schn\"urer; F. Schulze; M.  Simon, {\sl  Stability of Euclidean space under Ricci
flow}, Comm. Anal. Geom. \textbf{16} (2008), no. 1, 127--158.

\bibitem[SSS2]{SS2} O. C. Schn\"urer; F. Schulze; M.  Simon, {\sl Stability of hyperbolic space under Ricci flow},  Comm. Anal. Geom. \textbf{19} (2011), no. 5, 1023-1047.

\bibitem[T]{T} Tam, L.-F.,  {\sl Exhaustion functions on complete manifolds},  211--215 in   Recent advances in geometric analysis,  Adv. Lect. Math. (ALM), \textbf{11}, Int. Press, Somerville, MA, 2010.

\bibitem[TY]{TY} G. Tian; S.-T.   Yau, {\sl Complete \K manifolds with zero Ricci curvature. I.}, J. Amer. Math. Soc. \textbf{3} (1990), 579--609.

    \bibitem[WZ]{WZ}H.-H Wu; F. Zheng, {\sl  Examples of positively curved complete \K manifold}, Geometry and Analysis Volume I, Advanced Lecture in Mathematics \textbf{17}, Higher Education Press and International Press, Beijing and Boston, 2010, pp. 517--542.

\bibitem[Y]{Y} B. Yang, {\sl On a problem of Yau regarding a higher dimensional generalization of the Cohn-Vossen inequality}, Math. Ann. \textbf{355}(2) (2013), 765-781.

\bibitem[YZ]{YZ} B. Yang; F. Zheng, {\sl $U(n)$-invariant \KR flow with non-negative curvature}, Comm. Anal. Geom .\textbf{ 21}, (2013)
no. 2, 251--294.
\bibitem[Yu]{Yu} Yu, C.-J., {\sl A note on Kahler-Ricci flow},  Math. Z.  \textbf{272}  (2012),  no. 1-2, 191--201.
\end{thebibliography}
\end{document}